\documentclass[reqno]{amsart}
\usepackage{amsfonts,amsmath,amsthm,amssymb,latexsym, cite}%}
\usepackage{graphicx}
\usepackage[colorlinks,plainpages,citecolor=magenta, linkcolor=blue]{hyperref}%,bookmarksnumbered,,backref
\newtheorem{thm}{Theorem}

\newtheorem{lem}{Lemma}
\newtheorem{prop}{Proposition}
\theoremstyle{definition}
\newtheorem{defn}{Definition}
\theoremstyle{remark}
\newtheorem{rem}{Remark}
\numberwithin{equation}{section}
\newcommand{\dom}{\mathop{\rm dom}}
\renewcommand{\Im}{\mathop{\rm Im}}

\renewcommand{\kappa}{\varkappa}

\newcommand{\Real}{\mathbb R}
\newcommand{\eps}{\varepsilon}
\newcommand{\cI}{\mathcal{I}}

\newcommand{\ra}{\rangle}
\newcommand{\la}{\langle}

\newcommand{\rd}[1]{{\color{red}{#1}}}
\renewcommand{\emph}[1]{{\textit{#1}}}
\renewcommand{\phi}{\varphi}

\newcommand\xe{\left(\tfrac x\eps\right)}

\newcommand\se{\left(\tfrac s\eps\right)}

\begin{document}

\title[Two-parametric $\delta'$-interactions]{
Two-parametric $\delta'$-interactions:\\ approximation  by Schr\"{o}dinger operators with localized rank-two perturbations
}%
\author{Yuriy Golovaty}%
\address{Department of Mechanics and Mathematics,
  Ivan Franko National University of Lviv\\
  1 Universytetska str., 79000 Lviv, Ukraine}%
\email{yuriy.golovaty@lnu.edu.ua}%



% ----------------------------------------------------------------
\begin{abstract}
  We construct  a norm resolvent approximation to the family of point interactions $f(+0)=\alpha f(-0)+\beta f'(-0)$, $f'(+0)=\alpha^{-1}f'(-0)$
by  Schr\"{o}\-din\-ger operators  with
loca\-li\-zed rank-two perturbations coupled with short range potentials. In particular, a new approximation to the  $\delta'$-inter\-actions is obtained.
\end{abstract}
\maketitle
\vskip30pt
\section{Introduction}

Schr\"{o}dinger operators with pseudo-potentials  that are distributions supported on  discrete sets (such potentials are usually termed point interactions) have received  considerable attention from many researchers  over several past decades. The point interactions have been widely and extensively investigated from various points of view and the study of solvable models based on the concept of  zero range quantum interactions has a long and interesting history. General references for this fascinating area are \cite{Albeverio2edition, AlbeverioKurasov}.
Historically the point interactions were introduced in quantum mechanics as  li\-mits of families of squeezed potentials. The main purpose was to find  solvable mo\-dels describing with admissible fidelity the real quantum processes  go\-ver\-ned by  Hamiltonians with localized potentials.  However the connection between real short-range interactions and point interactions is very complex and ambiguously determined.
This is  certainly the reason why the ``inverse'' problem -- how to approximate a given point interaction by regular Hamiltonians with localized perturbations -- is also important.

In the one-dimensional case, among all zero range  interactions, the $\delta'$-interactions, along with $\delta$ potentials, are most studied in this kind of research.   The $\delta'$-interaction at the ori\-gin, of strength $\beta$, is described by the self-adjoint operator $f\mapsto -f''$ in $L_2(\Real)$
restricted to functions in $W_2^2(\Real\setminus\{0\})$ obeying the interface conditions
\begin{equation}\label{ClassicDI}
\begin{pmatrix} f(+0) \\ f'(+0) \end{pmatrix}
=
\begin{pmatrix}
  1 &  \beta\\
  0 &  1
\end{pmatrix}
\begin{pmatrix} f(-0) \\ f'(-0) \end{pmatrix}.
\end{equation}
This operator is widely accepted as a model for the pseudo-Hamiltonian
\begin{equation*}
  H=-\frac{d^2}{dx^2}+\beta \la \delta'(x),\,\cdot\,\ra\,\delta'(x).
\end{equation*}
However, as shown in \cite{AlbeverioKoshmanenkoKurasovNizhnik2002}, no self-adjoint regularization $-\frac{d^2}{dx^2}+\beta \la \phi_\eps,\,\cdot\,\ra\,\phi_\eps$ of $H$ provides an approximation to point interactions \eqref{ClassicDI}.
Here the sequence of smooth functions $\phi_\eps$ converges  in the sense of distributions to the first derivative of  the Dirac delta function.
\v{S}eba \cite{SebRMP} was the first to  approximate the $\delta'$-interactions in strong resolvent sense by the ope\-rators  $-\frac{d^2}{dx^2}+\lambda_\eps \la \phi_\eps,\,\cdot\,\ra\,\phi_\eps$ with an infinitely small $\eps\to 0$ coupling constant $\lambda_\eps$.
This result can be improved to convergence in the norm resolvent sense; see for instance \cite{ExnerManko2014}, where the problem on metric graphs was stu\-died. Families of non-self-adjoint Schr\"{o}dinger operators
with nonlocal perturbations which converge to the $\delta'$-interactions were constructed by Albeverio and Nizhnik~\cite{AlbeverioNizhnik2000}.

Exner, Neidhardt and Zagrebnov \cite{KlauderPhenomenon:2001} obtained very subtle potential approximations to the $\delta'$-interactions in the norm resolvent topology: the family of potentials was built  as a triple of $\delta$-like potentials, shrinking to the origin, with non-trivial dependence between coupling constants and  separation distances. The paper became a mathematical justification of the result previously obtained by Cheon and Shigehara \cite{CheonShigehara1998}, who built the approximation in terms of three $\delta$ functions with the renormalized strengths and  disappearing distances.
In the context of the ``three delta approximation'', it is worth mentioning the works of  Albeverio, Fassari and Rinaldi~\cite{FassariRinaldi2009, AlbeverioFassariRinaldi2015} and the recent publication of Zolotaryuk~\cite{ZolotaryukThreeDelta17}; see also \cite{CheonExner2004} for the case of quantum graphs.
The reader also interested in the literature on other aspects of $\delta'$-interactions and approximations of point interactions by local and  non-local perturbations is referred to
\cite{NizhFAA2003, NizhFAA2006, AlbeverioNizhnik2006, AlbeverioFassariRinaldi2013,  KuzhelZnojil, AlbeverioNizhnik2007, AlbeverioNizhnik2013, Lange2015}.

In this paper, we study Schr\"{o}dinger operators with
loca\-li\-zed rank-two perturbations coupled with short range $\delta$-like potentials. A careful asymptotic analysis of these operators shows that some part of the set of limit operators which can be obtained in the norm resolvent topology, as the support of perturbation shrinks to the origin,  deals with the $2$-parametric family of point interactions
\begin{equation}\label{GeneralizedDI}
\begin{pmatrix} f(+0) \\ f'(+0) \end{pmatrix}
=
\begin{pmatrix}
  \alpha &  \beta\\
  0 &  \alpha^{-1}
\end{pmatrix}
\begin{pmatrix} f(-0) \\ f'(-0) \end{pmatrix}.
\end{equation}
In fact, we built a norm resolvent approximation to these  point interactions. In particular, we obtained a new approximation to the classic  $\delta'$-interaction that corresponds to the case  $\alpha=1$.

It is worth to note that the $\delta'$-interactions should not be confused with the $\delta'$ potentials.  The Schr\"{o}dinger operators with $(a\delta'+b\delta)$-like potentials
\begin{equation}\label{DeltaPrimePotentials}
\mathcal{H}_\eps=-\frac{d^2}{dx^2}+\frac{1}{\eps^{2}} \, V_2\left(\frac{x}{\eps}\right)+\frac{1}{\eps}\, V_1\left(\frac{x}{\eps}\right)
\end{equation}
have been recently   investigated in \cite{ChristianZolotarIermak03, Zolotaryuk08, GolovatyMankoUMB, GolovatyHrynivJPA2010, GolovatyMFAT2012, GolovatyIEOT2013, GolovatyHrynivProcEdinburgh2013,ToyamaNogami, Zolotaryuks2011}.
For instance, it has been proved in \cite{GolovatyIEOT2013} that $\mathcal{H}_\eps$ converge, as $\eps\to 0$, in the norm resolvent sense to the operator $f\mapsto -f''$ in $L_2(\Real)$
restricted to functions in $W_2^2(\Real\setminus\{0\})$ such that
\begin{equation}\label{DeltaPrimePtn}
\begin{pmatrix} f(+0) \\ f'(+0) \end{pmatrix}
=
\begin{pmatrix}
  \mu &  0\\
  \nu &  \mu^{-1}
\end{pmatrix}
\begin{pmatrix} f(-0) \\ f'(-0) \end{pmatrix},
\end{equation}
if potential $V_2$ possesses a zero energy resonance, and to  the direct sum $S_-\oplus S_+$ of the  half-line Schr\"odinger operators
$S_\pm= -d^2/d x^2$ on~$\Real_\pm$ subject to the Dirichlet boundary condition at the origin, otherwise. The spectral properties of models
with point interactions \eqref{DeltaPrimePtn} as well as the scattering coefficients were studied in \cite{GadellaNegroNietoPL2009, GadellaGlasserNieto2011}.

% ----------------------------------------------------------------
\section{Statement of Problem and Main Result}
Let us consider the Schr\"{o}dinger operator
\begin{equation*}
  S_0=-\frac{d^2}{dx^2}+V(x)
\end{equation*}
in $L_2(\Real)$,  where potential $V$ is a real-valued, measurable and locally bounded. We also assume that $V$ is bounded from below in $\Real$.
Let $\phi_1$ and $\phi_2$ be real functions of compact support in $L_2(\Real)$. We introduce the rank-two operators
\begin{equation*}
  (B_\eps v)(x)=\phi_1\xe\int_\Real \phi_2\se v(s)\,ds
 +\phi_2\xe\int_\Real  \phi_1\se v(s) \,ds
\end{equation*}
acting in $L_2(\Real)$, and the family of self-adjoint operators
\begin{equation*}
 S_\eps= S_0+\eps^{-3}B_\eps+\eps^{-1}q\xe.
\end{equation*}
Here $q$ is also an  real-valued, measurable and bounded function  of compact support. The perturbation of  ope\-ra\-tor $S_0$ has a small support shrinking to the origin as the small positive parameter $\eps$ goes to zero. For this reason, $\dom S_\eps=\dom S_0$.

From now on, the inner scalar product and norm in $L_2(\Real)$ will be denoted by $\langle\cdot,\cdot \rangle$ and $\|\cdot\|$ respectively.
We denote by
\begin{equation*}
    f^{(-1)}(x)=\int_{-\infty}^x f(s)\, ds, \qquad   f^{(-2)}(x)=\int_{-\infty}^x (x-s)f(s)\, ds
\end{equation*}
the first and second  antiderivatives of a function $f$. The antiderivatives are well-defined for measurable functions of compact support, for instance. In addition, if $f$ has zero mean, then $f^{(-1)}$ is also a function of compact support.
In this paper, we will consider only the case when $\phi_1$ and $\phi_2$ are functions of zero means, i.e.,
  \begin{equation}\label{ZeroMean}
    \int_\Real \phi_j\,dx=0,\qquad j=1,2.
  \end{equation}
Therefore $\phi_1^{(-1)}$ and $\phi_2^{(-1)}$ have compact supports and
the function
\begin{equation}\label{Omega}
  \omega=\|\phi_2^{(-1)}\| \cdot\phi_1^{(-2)}-\|\phi_1^{(-1)}\| \cdot\phi_2^{(-2)}
\end{equation}
is constant in some neighbourhoods of  negative and positive infinities (see Fig.~\ref{FigHBS}).

\begin{figure}[hb]
  \centering
   \includegraphics[scale=0.7]{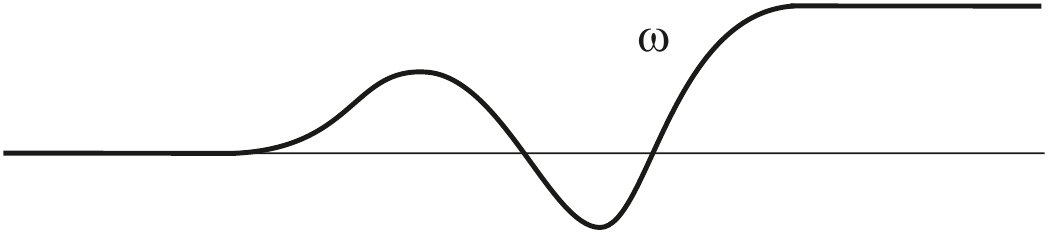}\\
  \caption{Plot of the function $\omega$.}\label{FigHBS}
\end{figure}

We introduce  notation
\begin{gather}\nonumber
\kappa=\lim_{x\to+\infty}\omega(x),
 \\\label{NotationAk}
 a_0=\int_\Real q\,dx,\qquad a_1=\int_\Real q\,\omega\,dx,\qquad a_2=\int_\Real q\,\omega^2\,dx.
\end{gather}
We will denote by $\mathcal{V}$ the space of $L_2(\Real)$-functions $f$ such that $f(x)=f_-(x)$ if $x<0$ and $f(x)=f_+(x)$ if $x>0$ for some $f_-$ and $f_+$ belonging to the domain of $S_0$.
Let us consider the operator $S_{\alpha\beta}f=-f''+Vf$,
\begin{equation*}
\dom S_{\alpha\beta}=\big\{f\in \mathcal{V}\colon f(+0)=\alpha f(-0)+\beta f'(-0),\quad f'(+0)=\alpha^{-1}f'(-0)\big\}.
\end{equation*}

Our main result is the following theorem.
\begin{thm}\label{MainTh}
Let $\phi_1$, $\phi_2$ and $q$  be integrable, real-valued functions with compact supports. Suppose that
\begin{itemize}
  \item[\textit{(i)}] $\phi_1$ and $\phi_2$ have zero means,
  antiderivatives $\phi_1^{(-1)}$ and $\phi_2^{(-1)}$ are orthogonal in $L_2(\Real)$, and
      \begin{equation}\label{n1n2}
     \|\phi_1^{(-1)}\|\cdot\|\phi_2^{(-1)}\|=1;
      \end{equation}
  \item[\textit{(ii)}]  potential $q$ satisfies conditions
  \begin{equation}\label{Qconds}
     a_0a_2=a_1^2, \qquad a_2\neq \kappa a_1.
  \end{equation}
\end{itemize}
Then the operator family $S_\eps$  converges as $\eps\to 0$ in the norm resolvent sense to operator $S_{\alpha\beta}$, where
\begin{equation}\label{AlphaBeta}
\alpha=\frac{a_2-\kappa a_1}{a_2}, \qquad \beta=\frac{\kappa^2}{a_2-\kappa a_1}.
\end{equation}
\end{thm}

There is a wide class of functions  $\phi_1$, $\phi_2$ and $q$ satisfying the assumptions in Theorem~\ref{MainTh}. Moreover, for any pair $(\alpha, \beta)$ of real numbers with $\alpha\neq 0$, there exists a family of operators $S_\eps$ such that $S_\eps\to S_{\alpha\beta}$ as $\eps\to 0$ in the norm resolvent sense. The sole exception is  the point interactions
with matrix
\begin{equation*}
\begin{pmatrix}
       \alpha & 0 \\
       0 & \alpha^{-1}
  \end{pmatrix}, \quad \alpha\neq 1,
\end{equation*}
because $\beta$ vanishes together with $\kappa$ and therefore $\alpha=1$ by \eqref{AlphaBeta}. Note that the last point interactions correspond to the case of $\delta'$-potentials and can be approximated by operators
\eqref{DeltaPrimePotentials} with $V_1=0$ provided $V_2$  possesses a zero energy resonance \cite{GolovatyHrynivJPA2010, GolovatyHrynivProcEdinburgh2013}.

It is a simple matter to choose $\phi_1$ and $\phi_2$.  For instance, it is enough to take two orthonormal in $L_2(\Real)$ functions $\eta_1$ and $\eta_2$ of compact support belonging to $W_2^1(\Real)$, and then set $\phi_1=\eta_1'$ and $\phi_2=\eta_2'$.
With these functions in hands, we can construct  $\omega$ and calculate $\kappa=\omega(+\infty)$. If $\kappa=0$, then $\alpha=1$ and $\beta=0$, and the limiting operator  is the free Schr\"{o}dinger operator on the line.
Suppose now that $\kappa$ is different from zero.  Note that $\omega= \eta_1^{(-1)}-\eta_2^{(-1)}$ is a continuous and non-constant function, by the orthogonality of  $\eta_1$ and $\eta_2$. Hence $1$, $\omega$, $\omega^2$ are linearly independent functions on each interval $[-r,r]$. Then for any  $(a_0,a_1,a_2)\in\Real^3$ there exists a potential $q$ of compact support for which equalities \eqref{NotationAk} hold.
Given $\alpha\neq1$ and $\beta\neq 0$, we choose $q$ such that
 \begin{equation*}
   a_0=\frac{(1-\alpha)^2}{\alpha\beta},\quad a_1=\frac{\kappa(1-\alpha)}{\alpha\beta},\quad
   a_2=\frac{\kappa^2}{\alpha\beta}.
 \end{equation*}
This triple of numbers satisfies \eqref{Qconds} and \eqref{AlphaBeta}.
The same is true for the case of the classic $\delta'$-interactions when $\alpha=1$ and $\beta\neq 0$, if we set
 $a_0=0$, $a_1=0$ and $a_2=\beta^{-1}\kappa^2$.

\begin{rem}
In this paper we do not consider the case $\kappa=0$, i.e.,  the limit operator is  the free Schr\"{o}dinger operator. From now on,  we will assume that $\kappa$ is different from zero.
\end{rem}

\begin{rem}
The orthogonality of $\phi_1^{(-1)}$ and $\phi_2^{(-1)}$ of course implies the linear independence of $\phi_1$ and $\phi_2$. Hence operator $B_\eps$ has actually rank two.
\end{rem}

\section{Half-Bound States}
Let us consider the operator
$$
   B=-\frac{d^2}{dx^2}+\langle \phi_2, \,\cdot\,\rangle\,\phi_1 +\langle \phi_1, \,\cdot\,\rangle\,\phi_2,\quad \dom B=W_2^2(\Real)
$$
in space $L_2(\Real)$.

\begin{defn}
We say that the  operator~$B$ possesses a half-bound state provided there exists a nontrivial solution~$\psi$ of the equation
\begin{equation}\label{EquationBu0}
-\psi''+\langle \phi_2, \psi\rangle\,\phi_1 +\langle \phi_1, \psi \rangle\,\phi_2=0
\end{equation}
that is bounded on the whole line.
\end{defn}

Let us introduce notation
\begin{equation*}
n_j=\big\|\phi_j^{(-1)}\big\|,\qquad  m_{j}=\int_\Real x \phi_j\,dx,
\qquad j=1,2.
\end{equation*}
Now we prove the first of two  key lemmas  for the proof of main theorem.

\begin{lem}\label{LemmaHBS}
Under  assumption \textit{(i)} of  Theorem~\ref{MainTh}, the operator $B$ possesses the $2$-dimensional space of half-bound states  generated by the constant function and function $\omega$, given by \eqref{Omega}.
\end{lem}

\begin{proof}
  Any solution of  \eqref{EquationBu0} can be written as
  \begin{equation*}
  \psi(x)=c_1\phi_1^{(-2)}(x)+c_2\phi_2^{(-2)}(x)+c_3 x+c_4,
  \end{equation*}
where the constants $c_k$ are connected via two conditions
\begin{equation}\label{HBSConds}
\begin{aligned}
  &\langle\phi_1^{(-2)},\phi_1\rangle \,c_1+ \bigl(\langle\phi_2^{(-2)},\phi_1\rangle -1\bigr)\,c_2+m_1c_3=0,
  \\
  &\bigl(\langle\phi_1^{(-2)},\phi_2\rangle -1\bigr)\,c_1+\langle\phi_2^{(-2)},\phi_2\rangle \,c_2+m_2c_3=0.
\end{aligned}
\end{equation}
These conditions can be easy derived from \eqref{EquationBu0} in view of the linear independence of $\phi_1$ and $\phi_2$.
In general, $\phi_1^{(-2)}$ and $\phi_2^{(-2)}$ do not belong to $L_2(\Real)$. But it will cause no confusion if we  use the scalar products $\langle\phi_i^{(-2)},\phi_j\rangle$ as notation for the integrals $\int_\Real \phi_i^{(-2)} \phi_j\,dx$,
which are finite, because of  compact supports of $\phi_j$.

Since $\psi=c_3 x+c_4$ in  some neighbourhood of  negative infinity, the constant $c_3$ must be zero, because we are looking for bounded solutions. Also, the constant function is a half-bound state, since $\phi_1$ and $\phi_2$ have zero means. Therefore if any other (linearly independent) half-bound state exists, then it  has the form
 \begin{equation*}
  \psi(x)=c_1\phi_1^{(-2)}(x)+c_2\phi_2^{(-2)}(x),
  \end{equation*}
where vector $\vec{c}=(c_1,c_2)$ must be a nontrivial solution of the linear system $A\vec{c}=0$ with matrix
\begin{equation*}
  A=
  \begin{pmatrix}
    \langle\phi_1^{(-2)},\phi_1\rangle  & \langle\phi_2^{(-2)},\phi_1\rangle -1\\
    \langle\phi_1^{(-2)},\phi_2\rangle -1 & \langle\phi_2^{(-2)},\phi_2\rangle
      \end{pmatrix}.
\end{equation*}
The system is obtained from \eqref{HBSConds} by putting $c_3=0$.

Since $\phi_1^{(-1)}$ and $\phi_2^{(-1)}$ are compactly supported, we  obtain
  \begin{equation*}
 \int_\Real \phi_i^{(-2)} \phi_j\,dx=\phi_i^{(-2)}\phi_j^{(-1)}\Big|_{-\infty}^{+\infty} -\int_\Real \phi_i^{(-1)} \phi_j^{(-1)}\,dx=-\la\phi_i^{(-1)},\phi_j^{(-1)}\ra.
  \end{equation*}
According to the assumptions,
antiderivative $\phi_1^{(-1)}$ and $\phi_2^{(-1)}$ are orthogonal in $L_2(\Real)$. From this we have
$\la\phi_1^{(-2)},\phi_2\ra=\la\phi_2^{(-2)},\phi_1\ra=0$,
 $\la\phi_1^{(-2)},\phi_1\ra=-n_1^2$ and $\la\phi_2^{(-2)},\phi_2\ra=-n_2^2$.
Hence
\begin{equation*}
  A=-
  \begin{pmatrix}
      n_1^2  & 1  \\
      1 & n_2^2
  \end{pmatrix}.
\end{equation*}
Matrix $A$ is degenerate by \eqref{n1n2} and thereby system $A\vec{c}=0$
admit a non trivial solution $\vec{c}=(n_2, -n_1)$. Hence, the function $\omega=n_2 \phi_1^{(-2)}-n_1 \phi_2^{(-2)}$ is also a half-bound state of $B$.
\end{proof}

\section{Auxiliary statements}

Without loss of generality we can assume that the supports of  $\phi_1$, $\phi_2$ and $q$ lie in interval $\cI=[-1,1]$. Then
\begin{equation}\label{Phi-1}
  \phi_j^{(-k)}(-1)=0
\end{equation}
for  $k=0,1,2$ and $j=1,2$.
Also,
\begin{equation}\label{Phi1}
  \phi_1^{(-1)}(1)=0, \quad  \phi_2^{(-1)}(1)=0, \quad
  \phi_1^{(-2)}(1)=-m_1, \quad \phi_2^{(-2)}(1)=-m_2,
\end{equation}
because from \eqref{ZeroMean} we have
\begin{equation*}
  \phi_j^{(-2)}(1)=\int_{-\infty}^{1}(1-x)\phi_j(x)\,dx=
  \int_{\Real}\phi_j(x)\,dx
  -\int_{\Real}x\phi_j(x)\,dx=-m_j.
\end{equation*}
Then  we also deduce that $\omega(-1)=\omega'(-1)=\omega'(1)=0$ and $\omega(1)=\kappa$, where
\begin{equation}\label{Kappa}
  \kappa=n_1m_2-n_2m_1.
\end{equation}

Next, a half-bound state $\psi$ of  $B$ is now  constant  outside $\cI$ as a bounded solution of equation $\psi''=0$, and therefore the restriction of $\psi$ to $\cI$ is a non-trivial solution of the  boundary value problem
\begin{equation}\label{NeumanProblem}
     -\psi''+(\phi_2, \psi)\,\phi_1 +( \phi_1, \psi)\,\phi_2=0,\quad t\in \cI, \qquad   \psi'(-1)=0, \; \psi'(1)=0,
\end{equation}
where  $(\cdot,\cdot)$ is the scalar product in $L_2(\cI)$.

Given $h\in L_2(\cI)$ and $a, b\in \mathbb{C}$, we consider the nonhomogeneous problem
\begin{equation}\label{NHbvp}
  -v'' +(\phi_2, v)\,\phi_1 +(\phi_1, v)\,\phi_2=h,\quad t\in\cI, \qquad
 v'(-1)=a, \; v'(1)=b.
\end{equation}
Owing to Lemma~\ref{LemmaHBS}, homogeneous problem \eqref{NeumanProblem} has a $2$-dimensional space of solutions. Therefore  problem \eqref{NHbvp} is in general unsolvable.

\begin{prop}\label{LemmaNHbvp}
Under assumption \textit{(i)} of  Theorem~\ref{MainTh},  the nonhomogeneous  boun\-dary value problem \eqref{NHbvp} admits a solution if and only if
\begin{equation}\label{SolvabilityOfNHP}
a-b=(1, h), \qquad
 a=(1-\kappa^{-1}\omega,h).
\end{equation}
Then among all solutions of \eqref{NHbvp} there exists a unique one such that
\begin{equation}\label{V(1)=0}
v(-1)=0, \qquad  v(1)=0.
\end{equation}
In addition, this solution satisfies  the estimate
\begin{equation}\label{EstV}
  \|v\|_{W_2^2(\cI)}\leq C\|h\|_{L_2(\cI)},
\end{equation}
where the constant $C$ does not depend on $h$.
\end{prop}

\begin{proof}
Conditions \eqref{SolvabilityOfNHP} can be easy obtained by multiplying equation \eqref{NHbvp} by $1$ and $\omega$ in turn and then integrating by parts twice in view of the boundary conditions. Though the sufficiency of \eqref{SolvabilityOfNHP} follows from the Fredholm alternative, we will prove it directly by explicit construction of the desired solution.

We look for a partial solution of \eqref{NHbvp} in the form
\begin{equation*}
   v_0=k_1\phi_1^{(-2)}+k_2\phi_2^{(-2)}-h^{(-2)}+at,
\end{equation*}
where $k_1$, $k_2$ are arbitrary constants and $h^{(-2)}(t)=\int_{-1}^t(t-s)h(s)\,ds$ .
Function $v_0$ satisfies boundary conditions \eqref{NHbvp} for all $k_1$ and $k_2$. In fact,
\begin{equation*}
  v_0'(-1)=k_1\phi_1^{(-1)}(-1)+k_2\phi_2^{(-1)}(-1)-h^{(-1)}(-1)+a=a,
\end{equation*}
by \eqref{Phi-1}. From \eqref{Phi1} and the first solvability condition in \eqref{SolvabilityOfNHP} we see
\begin{equation*}
  v_0'(1)=k_1\phi_1^{(-1)}(1)+k_2\phi_2^{(-1)}(1)-h^{(-1)}(1)+a=a-(1,h)=b,
\end{equation*}
since $h^{(-1)}(1)=(1,h)$.
Direct substitution $v_0$ into equation \eqref{NHbvp} yields
\begin{equation*}
  \begin{pmatrix}
      n_1^2  & 1  \\
      1 & n_2^2
  \end{pmatrix}
  \begin{pmatrix}
    k_1\\k_2
  \end{pmatrix}
  =
  \begin{pmatrix}
   am_1-(\phi_1, h^{(-2)})  \\
   am_2-(\phi_2, h^{(-2)})
  \end{pmatrix},
\end{equation*}
(cf. the proof of Lemma~\ref{LemmaHBS}).  According to \eqref{n1n2} we have  $n_1n_2=1$, and  then the system can be written as
\begin{equation}\label{LinSysK}
  \begin{pmatrix}
      n_1  & n_2  \\
      n_1 & n_2
  \end{pmatrix}
  \begin{pmatrix}
    k_1\\k_2
  \end{pmatrix}
  =
  \begin{pmatrix}
   g_1  \\
   g_2
  \end{pmatrix},
\end{equation}
where $g_1=n_2(am_1-(\phi_1, h^{(-2)}))$ and $g_2=n_1(am_2-(\phi_2, h^{(-2)}))$. Therefore the system is consistent if and only if $g_1=g_2$. But this equality is equivalent to the second solvability condition in \eqref{SolvabilityOfNHP}. Indeed, recalling now \eqref{Kappa}, we have
\begin{multline*}
  g_2-g_1=a(n_1m_2- n_2m_1)+(n_2\phi_1-n_1\phi_2, h^{(-2)})\\=a \kappa +(\omega'',h^{(-2)})
  = \kappa \big(a-(1-\kappa^{-1}\omega,h)\big)=0,
\end{multline*}
because integrating by parts twice gives
\begin{equation*}
(\omega'',h^{(-2)})=-\kappa h^{(-1)}(1)+(\omega, h)=-\kappa (1,h)+(\omega, h)=-\kappa (1-\kappa^{-1}\omega,h).
\end{equation*}
The vector $\vec{k}=(n_2g_1,0)$ solves  \eqref{LinSysK} and then
$v_0=n_2g_1\phi_1^{(-2)}-h^{(-2)}+at$
is a solution of \eqref{NHbvp}.
With the aid of $v_0$ we can construct a solution $v$ satisfying conditions \eqref{V(1)=0}. We set $v=v_0+a-\kappa^{-1}(v_0(1)+a)\,\omega$, i.e.,
\begin{equation*}
  v=  n_2g_1\phi_1^{(-2)}-h^{(-2)}+a(t+1)
  -\tfrac1{\kappa}\left(n_2m_1g_1-h^{(-2)}(1)+2a\right)\omega.
\end{equation*}

Estimate \eqref{EstV} looks strange  at first sight, because a solution of \eqref{NHbvp} is bounded by the right-hand side $h$ of the equation only without regard for right-hand sides $a$ and $b$ in the boundary conditions. But by virtue of solvability  conditions \eqref{SolvabilityOfNHP}, numbers $a$ and $b$ can be expressed via function $h$:
\begin{equation}\label{ABonH}
 a(h)=(1-\kappa^{-1}\omega,h), \qquad b(h)=-\kappa^{-1}(\omega, h).
\end{equation}
Therefore for each $h\in L_2(\cI)$ there exists a unique boundary data $(a(h),b(h))$ such that problem \eqref{NHbvp} is solvable.
So regarding  $a$ and $g_1$ as linear functionals in $L_2(\cI)$, we have the bounds
\begin{align*}
  &|a(h)|=\left|(1-\kappa^{-1}\omega,h)\right|\leq C_1\|h\|_{L_2(\cI)}, \\
  &|g_1(h)|=|n_2|\cdot|m_1a(h)-(\phi_1, h^{(-2)})|\leq C_2(|a(h)|+\|h^{(-2)}\|_{L_2(\cI)})\leq C_3\|h\|_{L_2(\cI)}.
\end{align*}
From this and explicit formula for $v$ we immediately deduce
\begin{equation*}
\|v\|_{W_2^2(\cI)}\leq C_4\left(|a(h)|+|g_1(h)|+\|h^{(-2)}\|_{W_2^2(\cI)}\right)\leq C_5\|h\|_{L_2(\cI)},
\end{equation*}
since the operator $L_2(\cI)\ni h\mapsto h^{(-2)}\in W^2_2(\cI)$ is bounded.
\end{proof}

In the end of the section,  we record some technical assertion.
Let $[g]_{a}$ denote  the jump $g(a+0)-g(a-0)$ of  function $g$ at a point $a$.

\begin{prop}\label{PropW22Corrector}
Let $U$ be the real line with two removed points $x=-\eps$ and $x=\eps$, i.e., $U=\Real\setminus \{-\eps,\eps\}$.
Assume that function $g\in W_{2, loc}^2(U)$ along with its first derivative has jump discontinuities at points $x=-\eps$ and $x=\eps$. There exists a function $\rho\in C^\infty(U)$ such that   $g+\rho$ belongs to $W_{2, loc}^2(\Real)$ and
    \begin{equation}\label{REst}
        |\rho^{(k)}(x)|\leq C \Bigl(\bigl|[g]_{-\eps}\bigr|+\bigl|[g]_{\eps}\bigr|
        +\bigl|[g']_{-\eps}\bigr|+\bigl|[g']_{\eps}\bigr|\Bigr)
    \end{equation}
    for $|x|\geq \eps$,  $k=0,1,2$, where the constant $C$ does not depend on $g$ and $\eps$. Moreover, $\rho$ is a function of compact support and $\rho$ vanishes in $(-\eps,\eps)$.
\end{prop}
\begin{proof}
Let us introduce functions $w_0$ and $w_1$ that are smooth outside the origin, have compact supports contained in $[0,\infty)$, and such that $w_0(+0)=1$, $w_0'(+0)=0$, $w_1(+0)=0$ and $w_1'(+0)=1$.
We set
\begin{equation*}%\label{CorrectoR}
\rho(x)=[g]_{-\eps}\, w_0(-x-\eps)-[g']_{-\eps}\,w_1(-x-\eps)\\
-[g]_{\eps}\,w_0(x-\eps)-[g']_{\eps}\,w_1(x-\eps).
\end{equation*}
By construction,  $\rho$ has a compact support and vanishes in $(-\eps,\eps)$. An easy computation also shows that
\begin{equation*}
  [\rho]_{-\eps}=-[g]_{-\eps}, \quad[\rho]_\eps=-[g]_\eps, \quad[\rho']_{-\eps}=-[g']_{-\eps}, \quad [\rho']_\eps=-[g']_\eps.
\end{equation*}
Therefore $g+\rho$ is continuous on~$\Real$ along with the first derivative and consequently belongs to $W_{2, loc}^2(\Real)$.
Finally, the explicit formula for $\rho$  makes it obvious that inequality \eqref{REst} holds.
\end{proof}

\section{Proof of Theorem \ref{MainTh}}
Given $f\in L_2(\Real)$ and $\zeta\in \mathcal{C}\setminus\Real$, we must compare two elements $ u_\eps=(S_\eps-\zeta)^{-1}f$ and $u=(S_{\alpha\beta}-\zeta)^{-1}f$, and show that the difference $u_\eps-u$ is infinitely small in $L_2(\Real)$, as $\eps\to 0$, uniformly on $f$. The basic idea of the proof is to construct a~suitable approximation to $u_\eps$.
For $\eps>0$, we introduce the sequence of functions
\begin{equation}\label{AsymptoticsUeps}
 y_\eps(x)=
  \begin{cases}
       u(x)   & \text{if }|x|>\eps,\\
      \psi\xe+\eps v_\eps\xe& \text{if }  |x|<\eps,
 \end{cases}
\end{equation}
where
$\psi(t)=u(-0)+\kappa^{-1}(u(+0)-u(-0))\,\omega(t)$ is a restriction of a half-bound state of operator $B$ such that
\begin{equation}\label{PsiAt1}
\psi(-1)=u(-0), \qquad \psi(1)=u(+0);
\end{equation}
function $v_\eps$ solves  the  problem
\begin{align}\label{bvpV1Eq}
  -&v_\eps'' +(\phi_2, v_\eps)\,\phi_1 +(\phi_1, v_\eps)\,\phi_2=
\eps f(\eps t) -q \psi(t),\quad t\in\cI, \\\label{bvpV1Cnds}
 &v_\eps'(-1)=u'(-0)+\xi_\eps(f), \quad v_\eps'(1)=u'(+0)+\eta_\eps(f)
\end{align}
with some numbers $\xi_\eps$ and $\eta_\eps$ depending on $f$.

First we record some estimates on $u$ and $\psi$. We  observe that $(S_{\alpha\beta}-\zeta)^{-1}$ is a bounded operator from~$L_2(\Real)$ to $\dom S_{\alpha\beta}$ equipped with the
graph norm. Since  potential $V$ is locally bounded, the latter space is a subspace of $W_{2,loc}^2(\Real\setminus\{0\}) \cap\mathcal{V}$. Hence there exists a constant independent of $f$ such that
\begin{equation*}
\|u\|_{W_2^2((-r,r)\setminus \{0\})}\leq c\|f\|,
\end{equation*}
for any $r>0$, and thus $\|u\|_{C^1([-r,0])}+\|u\|_{C^1([0,r])}\leq c\|f\|$, by the Sobolev embedding theorem. In particular, we  have
\begin{equation*}
|u(-0)|+|u(+0)|+|u'(-0)|+|u'(+0)|\leq c\|f\|.
\end{equation*}
It follows from the last bound that
\begin{equation}\label{EstPsi}
  \|\psi\|_{L_2(\cI)}\leq c_1\big(|u(-0)|+|u(+0)|\big)
  \leq c_2 \|f\|.
\end{equation}
Next, there exists a constant  being independent of $\eps$ and $u$ such that
\begin{equation}\label{EstU(Eps)-U(0)}
\big|u^{(k)}(-\eps)-u^{(k)}(-0)\big|+
\big|u^{(k)}(\eps)-u^{(k)}(+0)\bigl|\leq C\eps^{1/2}\|f\|
\end{equation}
for $k=0,1$, since
\begin{equation*}
\bigr|u^{(k)}(\pm\eps)-u^{(k)}(\pm 0)\bigl|\leq \left|\int_0^{\pm\eps}|u^{(k+1)}(x)|\,dx\right|
   \leq c\eps^{1/2} \|u\|_{W_2^2((-1,1)\setminus \{0\})}.
\end{equation*}

In view of Proposition~\ref{LemmaNHbvp}, for each $f\in L_2(\Real)$ there exists a unique pair $(\xi_\eps,\eta_\eps)$ such that
problem \eqref{bvpV1Eq}, \eqref{bvpV1Cnds} admits a solution.
We conclude from \eqref{ABonH} that
\begin{equation}\label{FunctionalsAB}
\begin{aligned}
&\xi_\eps(f) =\big(\kappa^{-1}\omega -1,q\psi\big)-u'(-0)+\eps\big(1-\kappa^{-1}\omega, f(\eps \,\cdot) \big), \\
     &\eta_\eps(f)=\kappa^{-1}\big(q\omega, \psi\big)-u'(+0)-\eps \kappa^{-1}\big(\omega,f(\eps \,\cdot)\big).
\end{aligned}
\end{equation}
Then \eqref{bvpV1Eq}, \eqref{bvpV1Cnds} has a solutions $v_\eps$ such that  $v_\eps(-1)=0$, $v_\eps(1)=0$ and
\begin{equation}\label{EstVeps}
  \|v_\eps\|_{W_2^2(\cI)}\leq
  c_1 \left(\|\psi\|_{L_2(\cI)}+\eps \|f(\eps\, \cdot)\|_{L_2(\cI)}\right)\leq c_2 \|f\|,
\end{equation}
by \eqref{EstV}. Here we employed \eqref{EstPsi} and the obvious inequality
\begin{equation}\label{EstF(eps)}
 \|f(\eps\,\cdot\,)\|_{L_2(\cI)}\leq c\eps^{-1/2} \|f\|.
\end{equation}

Function~$y_\eps$ given by \eqref{AsymptoticsUeps} does not
belong to the domain of $S_\eps$, because it is in general discontinuous at  points $x=-\eps$ and $x=\eps$.  Although $y_\eps$ has points of  discontinuity, we will show that its jumps  and  jumps of its first derivative at these points are small as $\eps\to 0$.
Recalling \eqref{PsiAt1}, boundary conditions \eqref{NeumanProblem} and \eqref{bvpV1Cnds}  we see at once that
\begin{equation}\label{JumpsOfYeps}
\begin{aligned}
  &[y_\eps]_{-\eps}=u(-0)-u(-\eps),&&
  [y_\eps']_{-\eps}=u'(-0)-u'(-\eps)+\xi_\eps(f),
  \\
  & [y_\eps]_{\eps}=u(\eps)-u(+0),&&
  [y_\eps']_{\eps} =u'(\eps)-u'(+0)- \eta_\eps(f).
\end{aligned}
\end{equation}

The following lemma is the second key point of the proof.
From the technical point of view, it states that the jumps of $y_\eps$ and $y'_\eps$ are small only for $q$ satisfying condition \eqref{Qconds} and $\alpha$, $\beta$ given by \eqref{AlphaBeta}.
But in essence, the lemma demonstrates  a subtle connection between the half-bound states of $B$, potential $q$ and point interactions \eqref{GeneralizedDI}.
\begin{lem}
Suppose that $u=(S_{\alpha\beta}-\zeta)^{-1}f$ with $\alpha$ and $\beta$ given by \eqref{AlphaBeta}. Under  the assumptions  of  Theorem~\ref{MainTh}, sequences $\xi_\eps(f)$ and $\eta_\eps(f)$ are infinitesimal as $\eps\to 0$ and the estimate
\begin{equation}\label{EstAB}
   |\xi_\eps(f)|+|\eta_\eps(f)|\leq C\eps^{1/2}\|f\|
\end{equation}
holds for all $f\in L_2(\Real)$ and  a constant $C$ which does not depend on $f$.
\end{lem}
\begin{proof}
We will show that the terms in \eqref{FunctionalsAB}, which do not depend on $\eps$, are equal to zero, i.e.,
\begin{equation}\label{ZeroTerms}
\big(\kappa^{-1}\omega -1,q\psi\big)-u'(-0)=0, \quad \kappa^{-1}\big(q\omega, \psi\big)-u'(+0)=0.
\end{equation}
Recall that $u$ is a unique solution of the problem
\begin{align}\label{LimitProblemEq}
  -&u''+(V-\zeta)u= f \; \text{in } \Real\setminus\{0\},\\\label{LimitProblemCnds}
   &u(+0)=\alpha u(-0)+\beta u'(-0), \quad u'(+0)=\alpha^{-1} u'(-0).
\end{align}
Then half-bound state $\psi$ in \eqref{AsymptoticsUeps} can be written as
\begin{equation}\label{PsiAnotherForm}
  \psi(t)=u(-0)+\kappa^{-1}\Big((\alpha-1) u(-0)+\beta u'(-0)\Big)\,\omega(t).
\end{equation}
First, we consider the case when constant $a_1$ in \eqref{Qconds} is different from zero. Then we also have $a_0\neq0$ and $a_2\neq0$. With this, we obtain
\begin{multline*}
  \big(q,\psi\big)=\Big(q,\,u(-0)+\kappa^{-1}\big((\alpha-1) u(-0)+\beta u'(-0)\big)\,\omega\Big)
  \\=
  a_0u(-0)+\kappa^{-1}a_1\big((\alpha-1) u(-0)+\beta u'(-0)\big)
 \\=
\left(\kappa^{-1}a_1(\alpha-1)+a_0\right)u(-0)+\kappa^{-1}a_1\beta u'(-0)
 \\=
 \frac{a_1}{\kappa}\left(\alpha-\frac{a_1-\kappa a_0}{a_1}\right)u(-0)+\frac{a_1\beta}{\kappa}\, u'(-0)=\frac{a_1\beta}{\kappa}\, u'(-0),
\end{multline*}
because
\begin{equation*}
  \alpha=\frac{a_2-\kappa a_1}{a_2}=\frac{a_0a_2-\kappa a_0 a_1}{a_0a_2}=\frac{a_1-\kappa a_0}{a_1}
\end{equation*}
in view of identity $a_0a_2=a_1^2$.
In the same manner, we deduce
\begin{multline*}
  \big(q\omega,\psi\big)=
  a_1u(-0)+\kappa^{-1}a_2\big((\alpha-1) u(-0)+\beta u'(-0)\big)
  \\=
 \frac{a_2}{\kappa}\left(\alpha-\frac{a_2-\kappa a_1}{a_2}\right)u(-0)+\frac{a_2\beta}{\kappa}\, u'(-0)=\frac{a_2\beta}{\kappa}\, u'(-0),
\end{multline*}
by the choice of $\alpha$ in \eqref{AlphaBeta}.
By the above, we have
\begin{multline*}
  \big(\kappa^{-1}\omega -1,q\psi\big)-u'(-0)
  =\kappa^{-1}\big(q\omega,\psi\big)-\big(q, \psi\big)-u'(-0)\\
    =\left(\frac{a_2\beta}{\kappa^2}-\frac{a_1\beta}{\kappa} -1\right)u'(-0)
    =\Big(\beta\cdot\frac{a_2-\kappa a_1}{\kappa^2}-1\Big)u'(-0)=0,
\end{multline*}
by the choice of $\beta$. Since $\alpha^{-1}=\kappa^{-2} a_2\beta$, we find
\begin{equation*}
  \kappa^{-1}\big(q\omega, \psi\big)-u'(+0)
  =
  \kappa^{-2} a_2\beta\, u'(-0)-u'(+0) = \alpha^{-1} u'(-0)-u'(+0)
 =0,
\end{equation*}
by the second boundary condition in \eqref{LimitProblemCnds}. Note that the first condition  \eqref{LimitProblemCnds}  is already used in \eqref{PsiAnotherForm}.
Therefore identities \eqref{ZeroTerms} hold and
\begin{equation}\label{ABshortform}
  \xi_\eps(f) =\eps\big(1-\kappa^{-1}\omega, f(\eps \,\cdot)\big),\qquad
  \eta_\eps(f)=
 -\eps \kappa^{-1}\big(\omega,f(\eps \,\cdot)\big).
\end{equation}
Finally then, from the last formulae and inequality \eqref{EstF(eps)} we immediately deduce estimate \eqref{EstAB}.

Now we consider the case $a_1=0$ which corresponds to the classic $\delta'$-interaction with $\alpha=1$ and $\beta=\kappa^{2}{a_2}^{-1}$.
Consequently \eqref{Qconds} implies $a_0=0$. Also,  \eqref{PsiAnotherForm} reduces to $\psi(t)=u(-0)+\kappa^{-1}\beta u'(0)\,\omega(t)$,
since $u'(-0)=u'(+0)=u'(0)$. A direct calculation shows that
$\big(q,\psi\big)=0$ and $\big(q\omega,\psi\big)=\kappa^{-1}a_2\beta\, u'(0)$. Therefore
\begin{align*}
&\begin{aligned}
\xi_\eps(f)&=\eps\big(1-\kappa^{-1}\omega, f(\eps \,\cdot)\big)+\kappa^{-1}\big(q\omega, \psi\big)-u'(0)\\
&=
  \eps\big(1-\kappa^{-1}\omega, f(\eps \,\cdot)\big)+\kappa^{-2}a_2\beta\, u'(0)-u'(0)=\eps\big(1-\kappa^{-1}\omega, f(\eps \,\cdot)\big),
\end{aligned}\\
&\begin{aligned}
   \eta_\eps(f)&=-\eps \kappa^{-1}\big(\omega,f(\eps \,\cdot)\big)+\kappa^{-1}\big(q\omega, \psi\big)-u'(0)\\
  &=
 -\eps \kappa^{-1}\big(\omega,f(\eps \,\cdot)\big)+
 \kappa^{-2} a_2\beta\, u'(0)-u'(0)=
 -\eps \kappa^{-1}\big(\omega,f(\eps \,\cdot)\big),
\end{aligned}
\end{align*}
and  $\xi_\eps(f)$ and $\eta_\eps(f)$ also have  the form \eqref{ABshortform}, which completes the proof.
\end{proof}

Returning now to the jumps \eqref{JumpsOfYeps}, we see at once that
\begin{equation*}
  \big|[y_\eps]_{-\eps}\big|+\big|[y_\eps]_{\eps}\big|
  +\big|[y_\eps']_{-\eps}\big|+\big|[y_\eps']_{\eps}\big|
  \leq C\eps^{1/2}\|f\|,
\end{equation*}
by \eqref{EstU(Eps)-U(0)} and \eqref{EstAB}.
Owing to Proposition~\ref{PropW22Corrector} there exists a corrector $\rho_\eps$ such that $Y_\eps=y_\eps+\rho_\eps$ belongs to $W_{2,loc}^2(\Real)$. Moreover, $\rho_\eps$ has a compact support, $\rho_\eps(x)=0$ for $x\in(-\eps,\eps)$, and
\begin{equation}\label{EstRhoEps}
  |\rho_\eps(x)|+|\rho_\eps''(x)|\leq C\eps^{1/2}\|f\|\quad \text{ for } |x|\geq \eps.
\end{equation}

Since $Y_\eps$ belongs to the domain of $S_\eps$, we now can compute $F_\eps=(S_\eps-\zeta)Y_\eps$ in order to estimate the accuracy of approximation.
From \eqref{LimitProblemEq} it follows  that
\begin{equation*}
  F_\eps(x)=\big( -\tfrac{d^2}{dx^2}+(V(x)-\zeta)\big)\big(u(x)+\rho_\eps(x)\big)
  =f(x)-\rho_\eps''(x)+(V(x)-\zeta)\rho_\eps(x)
\end{equation*}
for all $x$ such that $|x|>\eps$.  In the case $|x|<\eps$, we have
\begin{equation*}
\begin{aligned}
   F_\eps&(x)=
     -\frac{d^2}{dx^2}\big(Y_\eps\xe\big)+(V(x)-\zeta) Y_\eps\xe\\
     &+\eps^{-3}\int\limits_{-\eps}^\eps \Big(
     \phi_1\xe\phi_2\se+\phi_2\xe\phi_1\se
     \Big)Y_\eps\se\,ds+\eps^{-1}q\xe Y_\eps\xe\\
     &= \eps^{-2} \Big(-\psi''\xe+( \phi_2,\psi)\, \phi_1\xe+( \phi_1, \psi) \,\phi_2\xe \Big)\\
     & +\eps^{-1} \Big(-v_\eps''\xe+( \phi_2, v_\eps )\, \phi_1\xe+( \phi_1, v_\eps) \,\phi_2\xe+ q\xe \psi\xe\Big)\\
     & +q\xe v_\eps\xe+(V(x)-\zeta)Y_\eps\xe
     =f(x)+q\xe v_\eps\xe+(V(x)-\zeta) Y_\eps\xe,
\end{aligned}
\end{equation*}
by \eqref{NeumanProblem} and \eqref{bvpV1Eq}.
Therefore  $(S_\eps-\zeta)Y_\eps=f+r_\eps$, where
\begin{equation*}
  r_\eps(x)=
      \begin{cases}
        -\rho_\eps''(x)+(V(x)-\zeta)\rho_\eps(x) & \text{if } |x|>\eps,\\
        q\xe v_\eps\xe+(V(x)-\zeta) Y_\eps\xe & \text{if } |x|<\eps.
      \end{cases}
\end{equation*}
For any $g\in L_2(\cI)$ we have
$\|g(\eps^{-1}\cdot)\|_{L_2(-\eps,\eps)}=\eps^{1/2}\|g\|_{L_2(\cI)}$.
Using this equality together with the facts that $V$ is local bounded, $\rho_\eps$ has a compact support, and $Y_\eps=y_\eps$ on $(-\eps,\eps)$, we obtain
\begin{multline*}
  \|r_\eps\|\leq  c_1 \big(\|\rho_\eps''+(\zeta-V) \rho_\eps\|+\|q(\eps^{-1}\cdot) v_\eps(\eps^{-1}\cdot)+(V-\zeta) Y_\eps(\eps^{-1}\cdot)\|_{L_2(-\eps,\eps)}\big)\\
  \leq c_2\max\limits_{|x|>\eps}(|\rho_\eps|+|\rho_\eps''|)+c_3 \eps^{1/2}\big(\|v_\eps\|_{L_2(\cI)}+\|y_\eps\|_{L_2(\cI)}\big)\\
  \leq c_2\max\limits_{|x|>\eps}(|\rho_\eps|+|\rho_\eps''|)
  +c_4\eps^{1/2}\big(\|\psi\|_{L_2(\cI)}+\|v_\eps\|_{L_2(\cI)}\big).
\end{multline*}
Combining  estimates \eqref{EstPsi}, \eqref{EstVeps} and \eqref{EstRhoEps} yields the bound
\begin{equation}\label{EstReps}
  \|r_\eps\|\leq c\eps^{1/2}\|f\|.
\end{equation}
We can also apply the similar considerations to the difference
\begin{equation*}
  Y_\eps(x)-u(x)=
  \begin{cases}
       \rho_\eps(x)   & \text{if }|x|>\eps,\\
      \psi\xe+\eps v_\eps\xe-u(x)& \text{if }  |x|<\eps
 \end{cases}
\end{equation*}
and obtain the estimate
\begin{equation}\label{EstPeps}
   \|Y_\eps-u\|\leq c\eps^{1/2}\|f\|.
\end{equation}
The last bound means that a non-zero contribution in the $L_2$-norm of $Y_\eps$, as $\eps\to 0$, is produced by  $u=(S_{\alpha\beta}-\zeta)^{-1}f$ only.
Next, from $(S_\eps-\zeta)Y_\eps=f+r_\eps$ we  have
\begin{equation*}
  (S_\eps-\zeta)^{-1}f=Y_\eps-(S_\eps-\zeta)^{-1}r_\eps.
\end{equation*}
Finally we conclude that
\begin{multline*}
    \|(S_\eps-\zeta)^{-1}f-(S_{\alpha\beta}-\zeta)^{-1}f\|
    =\|Y_\eps-u-(S_\eps-\zeta)^{-1}r_\eps\|
      \\ \leq\|Y_\eps-u\|+ \|(S_\eps-\zeta)^{-1}r_\eps\|
    \leq\|Y_\eps-u\|+ |\Im\zeta|^{-1}\|r_\eps\|  \leq C \eps^{1/2}\|f\|,
\end{multline*}
by  \eqref{EstReps} and \eqref{EstPeps}. The last bound establishes
the norm resolvent convergence of  $S_\eps$ to  operator $S_{\alpha\beta}$ with $\alpha$ and $\beta$  given by \eqref{AlphaBeta}, which is the desired conclusion.

% ----------------------------------------------------------------
\bibliographystyle{amsplain}

\end{document}